\documentclass[11pt]{article}
\usepackage{dsfont}
\usepackage{amsmath, amsthm, amssymb, mathrsfs, eufrak}
\usepackage[abbrev,non-sorted-cites]{amsrefs}
\usepackage[all]{xy}
\usepackage{CJK}
\usepackage{grffile}
\usepackage{graphicx}
\newtheorem{thm}{Theorem}[section]
\newtheorem{cor}[thm]{Corollary}
\newtheorem{lem}[thm]{Lemma}

\newtheorem{definition}[thm]{Definition}
\newtheorem{rmk}[thm]{Remark}

\newtheorem{prop}[thm]{Proposition}
\newtheorem{conj}[thm]{Conjecture}

\begin{document}

\title{Reduced Open Gromov-Witten Invariants on K3 Surfaces and Multiple Cover Formula}
\author{Yu-Shen Lin}

\maketitle

\begin{abstract}
  In the paper, we study the wall-crossing phenomenon of reduced open Gromov-Witten invariants on K3 surfaces with rigid special Lagrangian boundary condition. As a corollary, we derived the multiple cover formula for the reduced open Gromov-Witten invariants. 
\end{abstract}

\section{Introduction}
Gromov-Witten Invariants naively count the number of (pseudo)-holomorphic curves of a given symplectic manifold possibly with certain prescribed conditions. Indeed, when the curve class is primitive, one can slightly perturb the almost complex structure such that all the pseudo-holomorphic curves are Fredholm regular. However, multiple covers of holomorphic curves occur when the curve classes are not primitive. In this situation, multiple covers appear and the associated moduli spaces can have the dimension higher then the expected dimension. In particular, multiple covers cause transversality difficulties. The definition of Gromov-Witten invariants will involve the techniques of virtual fundamental cycles and referred as the "virtual counting" of holomorphic curves. The Gromov-Witten invariants are usually rational numbers due to the appearance of automorphisms of the holomorphic curves. It is not clear that what is the geometric meaning of these rational numbers and their relation with the "counting" of the holomorphic curves in the target Calabi-Yau manifolds. Part of the Gopakumar-Vafa conjecture lights up a way to explain these numbers geometrically: one can rearrange the generating function of the Gromov-Witten invariants $N^g_d$ of genus $g$ and degree $d$ on a Calabi-Yau $3$-fold as follows: 
   \begin{align}
\sum_{\beta\neq 0}\sum_{g\geq 0}N^g_dt^{2g-2}q^d=\sum_{\beta\neq 0}\sum_{g\geq 0}n^g_d\sum_{k>0}\frac{1}{k}\big(2\sin{\frac{kt}{2}}\big)^{2g-2}q^{kd}.
\end{align} 
Then the numbers $n^g_d$, which contains the equivalent information as teh Gromov-Witten invariants, magically become integers. In particular, one has the following transformation for $g=0$ Gromov-Witten invariants $N^0_d$ and the integer invariants $n^0_d$,
\begin{align} \label{998}
   N^0_d=\sum_{k|d}n^0_{\frac{d}{k}}k^{-3},
\end{align} which is usually known as the Aspinwall-Morrison formula \cite{AM}\cite{LLY}\cite{V}. Similar viewpoint also appears in Taubes' work, which relates the Gromov invariants and Seiberg-Witten invariants \cite{T3}.

One may also try to consider the open string analogue Gromov-Witten invariants by naively counting Riemann surfaces with boundaries inside a symplectic manifolds and certain boundary condition. Similarly to the original situation, there are compactified moduli spaces of holomorphic curves with boundaries. However, the relevant compactified moduli spaces generally admit real codimension one boundaries and thus make the virtual fundamental classes not well-defined. As a consequence, the open Gromov-Witten invariants are usually only defined case by case. Especially for the case when the target spaces are compact Calabi-Yau manifolds, the open Gromov-Witten invariants are defined in two cases in the literature: the target space admits an anti-symplectic involution and the Lagrangian is the fixed locus \cite{S2}. Or the target space is a symplectic Calabi-Yau $3$-fold and the Lagrangian is a homological $3$-sphere \cite{F2}. 

If we restrict ourselves to Calabi-Yau $2$-folds, the K3 surfaces, then there is another difficulty arising from the dimension reason. The virtual dimension of the moduli spaces of holomorphic discs in a K3 surface and Lagrangian boundary condition is minus one. In other words, there is no holomorphic discs in a K3 surface with respect to the generic almost complex structures. Therefore, even the issue of real codimension boundary of the moduli space mentioned above can be resolved, we will expect the invariant to be zero. Similar to the idea of reduced Gromov-Witten invariants, the author \cite{L4} introduces an auxiliary $1$-parameter family of complex structures on the K3 surfaces and consider the moduli space of holomorphc discs in such $1$-parameter family. Although the new moduli space is topologically the same as the original one, the naturally equipped Kuranishi structures are different. One may view the procedure as a symplectic analogue of changing tangent-obstruction theory for defining the reduced Gromov-Witten invariants in algebraic geometry. In the paper \cite{L4}, the author mainly studies the case when the situation when the boundaries of the holomorphic discs are not homologous to zero. In this case, the special Lagrangian boundary conditions behave like certain stability conditions and there will be wall-crossing phenomenon similar to Donaldson-Thomas invariants when the Lagrangian boundary conditions are changing. 

In this paper, we will study the case when the boundary of the holomorhpic discs are homologous to the zero class. In this case, the bubbling phenomenon doesn't really contribute to the wall-crossing formula for the reduced open Gromov-Witten invariants. However, another kind of real codimension one boundaries of the moduli spaces appear: the pseudo-holomorphic discs may degenerate to pseudo-holomorphic spheres with a marked points hitting the Lagrangian submanifold. Therefore, the corresponding reduced open Gromov-Witten invariants will jump when certain parameter vary and can be resembled from the closed reduced Gromov-Witten invariants \cite{L}\cite{MP}. Using the multiple cover formula of reduced Gromov-Witten invariants of genus zero \cite{MP}, we derive the multiple cover formula for holomorhpic discs, as predicted in \cite{FOOO}\cite{F2}.

The paper is organized as follows: In section two, we review the basic knowledge of hyperK\"ahler manifold and holomorphic discs in twistor family. In section three, we define the reduced open Gromov-Witten invariants and study its wall-crossing formula. As a corollary, we derive the multiple cover formula for reduced open Gromov-Witten invariants via the wall-crossing formula.

\section*{Acknowledgement} The author would like to thank Shing-Tung Yau for constant support and encouragement. The author would also like to thank Jun Li and Zhiyuan Li for helpful discussion.

\section{Holomorphic Discs in Twistor Family}
 We will have a review of hyperK\"ahler geometry and holomorphic discs in the twistor family in this section. 
 \subsection{HyperK\"ahler Geometry and HyperK\"ahler Rotation}
 \begin{definition}
 A K3 surface is a compact complex surface with vanishing first fundamental group and vanishing first Chern class.
 \end{definition}
 
 \begin{rmk}\label{55}
    \cite{S3} Every K3 surface is K\"ahler.
 \end{rmk}
 From the vanishing of the first Chern class and Remark \ref{55}, the holonomy of a K3 surface falls in $SU(2)=Sp(1)$. In particular, every K3 surface is hyperK\"ahler. More precisely, let $X$ be a K3 surface and $\Omega$ be a nowhere vanishing holomorphic $(2,0)$-form guaranteed by the vanishing of first Chern class. Given any K\"ahler class $[\omega]$, there exists a Ricci-flat metric $\omega$ in the K\"ahler class $[\omega]$ such that 
   \begin{align*}
      \omega^2=c\Omega\wedge \bar{\Omega},
   \end{align*} where $c\in \mathbb{R}_{>0}$ from the Calabi conjecture \cite{Y1}. We will scale $\Omega$ (but still denote it $\Omega$) with an overall constant such that $c=\frac{1}{2}$. We will call $(\omega,\Omega)$ a hyperK\"ahler triple. A hyperK\"ahler triple $(\omega,\Omega)$ determines three almost complex structures $J_1,J_2,J_3$ via the relation:
   \begin{enumerate}
     \item 
          $\omega(\cdot,\cdot)=g(J_3\cdot,\cdot) $ and
        
      \item   
         $\Omega(\cdot,\cdot)=g(J_1\cdot,\cdot)+ig(J_2\cdot,\cdot)$
         is the holomorphic $2$-form with respect to the complex structure $J_3$.
   \end{enumerate} It can be checked that the three almost complex structures are actually integrable and satisfy the quaternionic relation. These induce a family of
 complex structures parametrized by $\mathbb{P}^1$, called the twistor line, on the underlying space $\underline{X}$ of $X$. Explicitly, they are given
 by
 \begin{equation*}
    J_{\zeta}=\frac{i(-\zeta+\bar{\zeta})J_1-(\zeta+\bar{\zeta})J_2+(1-|\zeta|^2)J_3}{1+|\zeta|^2},
    \hspace{3mm}
    \zeta\in \mathbb{P}^1.
 \end{equation*}
  The holomorphic
 symplectic $2$-forms $\Omega_{\zeta}$ with respect to the compatible
 complex structure $J_{\zeta}$ is given by
 \begin{equation}  \label{38}
   \Omega_{\zeta}=-\frac{i}{2\zeta}\Omega+\omega-\frac{i}{2}\zeta\bar{\Omega}.
   \end{equation} In particular, straightforward computation from (\ref{38}) gives
 \begin{prop} \label{49}
 Assume $\zeta=e^{i\vartheta}$, then we have
   \begin{align} \label{301}
      \omega_{\vartheta}:=\omega_{\zeta}&=-\mbox{Im}(e^{-i\vartheta}\Omega),\\
      \Omega_{\vartheta}:=\Omega_\zeta&=\omega-iRe(e^{-i\vartheta}\Omega).
   \end{align}
 \end{prop}
 
 \begin{rmk} \label{300}
    Let $L$ be a holomorphic Lagrangian in $(\underline{X},\omega,\Omega)$, namely,
    a complex submanifold of $X$ with $\mbox{dim}_{\mathbb{R}}L=\mbox{dim}_{\mathbb{C}}X$ and $\Omega|_L=0$. Assume that the
 north and south pole of the twistor line correspond to the hyperK\"ahler triple
 $(\omega,\Omega)$ and $(-\omega,\bar{\Omega})$ respectively, making
 $L$ a holomorphic Lagrangian. The hyperK\"ahler structures
 corresponding to the equator $\{\zeta=e^{i\vartheta}:|\zeta|=1\}$
 make $L$ a special Lagrangian in
 $X_{\vartheta}=(\underline{X},\omega_{\vartheta},\Omega_{\vartheta})$, i.e. $\omega_{\vartheta}|_{L}=\mbox{Im}\Omega_{\vartheta}|_{L}=0$ by Proposition \ref{49}.
 In particular, if $(\underline{X},\omega,\Omega)$ admits a holomorphic
 Lagrangian fibration, $X_{\vartheta}$ admits a special Lagrangian fibration structure, for each $\vartheta\in S^1$.
 This is the so-called hyperK\"ahler rotation trick.
 \end{rmk}
 \subsection{Holomorphic Discs in Twistor Family}
   Let $(X,\omega,\Omega)$ be a K3 surface with a choice of holomorphic volume form $\Omega$ and a Ricci-flat metric $\omega$ satisfying $2\omega^2=\Omega\wedge \bar{\Omega}$. Let $L$ be a holomorphic Lagrangian submanifold in $X$. From Remark \ref{300}, this induces an $S^1$-family of hyperK\"ahler manifolds $\mathfrak{X}^{[\omega]}=\{X_{\vartheta}\}$\footnote{Notice that the family $\{X_{\vartheta}\}$ does depend on the choice of $[\omega]$. However, we will just omit the subindex $[\omega]$ for simplicity.} containing $L$ as a special Lagrangian submanifold. Let $\gamma\in H_2(X,L)$ be a relative class and $\mathcal{M}_{\gamma}(X_{\vartheta},L)$ denotes the moduli space of stable discs holomorphic with respect to the complex structure of $X_{\vartheta}$, with boundary on $L$ and relative class $\gamma$. From the index theorem, the virtual dimension of the moduli space is negative. This suggests that  there is no pseudo-holomorphic discs in a K3 surface with special Lagrangian boundary condition with respect to a generic almost complex structure. Since we start with the data $(X,\omega,\Omega,L)$, there is no any single $\vartheta\in S^1$ is distinguished from the others. It is more natural to consider the following family version of moduli space 
     \begin{align*}
       \mathcal{M}_{\gamma}(\mathfrak{X}^{[\omega]},L):=\bigcup_{\vartheta\in S^1} \mathcal{M}_{\gamma}(X_{\vartheta},L),
     \end{align*} which is the moduli space of the stable discs holomorphic with respect to some complex structures of $X_{\vartheta}$ for some $\vartheta\in S^1$.
   A priori, the new moduli space $\mathcal{M}_{\gamma}(\mathfrak{X}^{[\omega]},L)$ may be complicated, say there might be infinitely many $\vartheta\in S^1$ such that $\mathcal{M}_{\gamma}(X_{\vartheta},L)\neq \emptyset$. Assume that $\mathcal{M}_{\gamma}(\mathfrak{X}^{[\omega]},L)\neq \emptyset$, then there exist a holomorphic disc in $X_{\vartheta}$ of relative class $\gamma$ for some $\vartheta$. In particular, we have 
      \begin{align*}
       0= \int_{\gamma}\mbox{Im}\Omega_{\vartheta}=-\int_{\gamma}\mbox{Re}(e^{-i\vartheta}\Omega) \\
       0< \int_{\gamma}\omega_{\vartheta}=-\int_{\gamma}\mbox{Im}(e^{-i\vartheta}\Omega).
      \end{align*} 
   In other words, if the moduli space $\mathcal{M}_{\gamma}(\mathfrak{X}^{[\omega]},L)\neq \emptyset$, then $\mathcal{M}_{\gamma}(X_{\vartheta},L)=\emptyset$ except $\vartheta=\mbox{Arg}\int_{\gamma}\Omega+\frac{\pi}{2}$. Although that $\mathcal{M}_{\gamma}(\mathfrak{X}^{[\omega]},L)\neq \emptyset$ topologically is the same as $\mathcal{M}_{\gamma}(X_{\vartheta},L)$ for some $\vartheta\in S^1$, its virtual dimension is zero and equipped naturally with a different Kuranishi structure \cite{L4}. This observation motivates the definition of the central charge in the next section.
       
   The moduli space $\mathcal{M}_{\gamma}(\mathfrak{X}^{[\omega]},L)$ a priori may admit real codimension one boundaries. There are two kinds of real codimension one strata of the boundaries: A holomorphic discs of relative class $\gamma\in H_2(X,L)$ can degenerate to a union of two discs of relative classes sum up to $\gamma$ and holomorphic to the same complex structures. This is usually called the disc bubbling phenomenon of holomorphic discs. We will call them the boundary of type I (see Figure \ref{1053}(a)) and denote it by  
    \begin{align}\label{912}
                   \partial_I \mathcal{M}_{\gamma}(\mathfrak{X}^{[\omega]},L)=
                  \cup \bigcup_{\substack{\gamma_1+\gamma_2=\gamma,\\ Z_{\gamma_1}/Z_{\gamma_2}\in \mathbb{R}_{>0} }}(\mathcal{M}_{1,\gamma_1}(\mathfrak{X}^{[\omega]},L)\; {}_{(ev_0,ev_{\vartheta})} \! \times_{(ev_0,ev_{\vartheta})}
                  \mathcal{M}_{1,\gamma_2}(\mathfrak{X}^{[\omega]},L))/\mathbb{Z}_2
               \end{align}  
   
   Now assume that $\gamma$ has its boundary $\partial \gamma=0\in H_1(L)$ homologous to zero. A holomorphic disc of relative class $\gamma$ (with no marked points) can "degenerate" to a rational sphere of class $\tilde{\gamma}$ with a marked points attached to the Lagrangian submanifold $L$, where $\tilde{\gamma}\in H_2(X)$ is a lifting of $\gamma\in H_2(X,L)$ via the exact sequence 
      \begin{align} \label{46}
           0\rightarrow H_2(L)\rightarrow H_2(X)\overset{\iota}{\rightarrow} H_2(X,L) \rightarrow H_1(L)=0.
       \end{align}
      We will refer this kind of the boundary as the boundary of type II \footnote{Notice that this is different from the sphere bubbling phenomenon, which is a degeneration of real codimension two.} (see Figure \ref{1053} (b)) and denote it by $\partial_{II} \mathcal{M}_{\gamma}(\mathfrak{X}^{[\omega]},L)$. In other words, 
      \begin{align} \label{5}
         {\mathcal{M}^{cl}_{\tilde{\gamma},1}(\mathfrak{X}^{[\omega]})}_{ev}\times L \subseteq \mathcal{M}_{\gamma}(\mathfrak{X}^{[\omega]},L)
      \end{align}as real codimension one boundary. Here $\mathcal{M}^{cl}_{\tilde{\gamma},1}(\mathfrak{X}^{[\omega]})$ denotes the moduli space of stable maps of rational curves into $\mathfrak{X}^{[\omega]}$ with one marked point and $ev$ is the evaluation map associated to the marked point. If $\tilde{\gamma}$ can be realized as holomorphic cycles in the $S^1$ family $\mathfrak{X}^{[\omega]}$, then we have $[\omega]\cdot \tilde{\gamma}=0$. Since any two homology classes $\tilde{\gamma}_i$ with $\iota_*(\tilde{\gamma_i})=\gamma$ is differed by a multiple of $[L]$. Therefore, as long as $[\omega]$ does not fall in any of the hyperplanes defined by $[\omega]\cdot \tilde{\gamma}=0$, with $\iota_*(\tilde{\gamma})=\gamma$, then $\mathcal{M}_{\gamma}(\mathfrak{X}^{[\omega]},L)$ will have no real codimension one boundary of the type II. A priori, there are infinite hyperplanes parametrized by $\tilde{\gamma}+kL$ may produce real codimension one boundary of type II.
      When $|k|$ is large, 
              \begin{align*}
           (\tilde{\gamma}+k[L])^2=\tilde{\gamma}^2+2k\tilde{\gamma}\cdot [L]+k^2[L]^2\sim -2k^2.
         \end{align*} The divisibility of $\tilde{\gamma}+k[L]$ is less than $k^2$ since each irreducible holomorphic curves in a K3 surface has self-intersection greater or equal to $-2$. Therefore, there will be only finitely \footnote{The argument of the paper is also valid for $L$ with higher genus than zero but there might be infinitely many valid hyperplanes.} many hyperplanes such that real codimension one boundary of type II can happen. We will call them valid hyperplanes and denoted by $W_{\tilde{\gamma}}$. The boundary of type II will play an essential role in the wall-crossing formula (\ref{15}) in this paper. 
    \begin{figure}\label{1053} 
             \begin{center}
             \includegraphics[height=3in,width=6in]{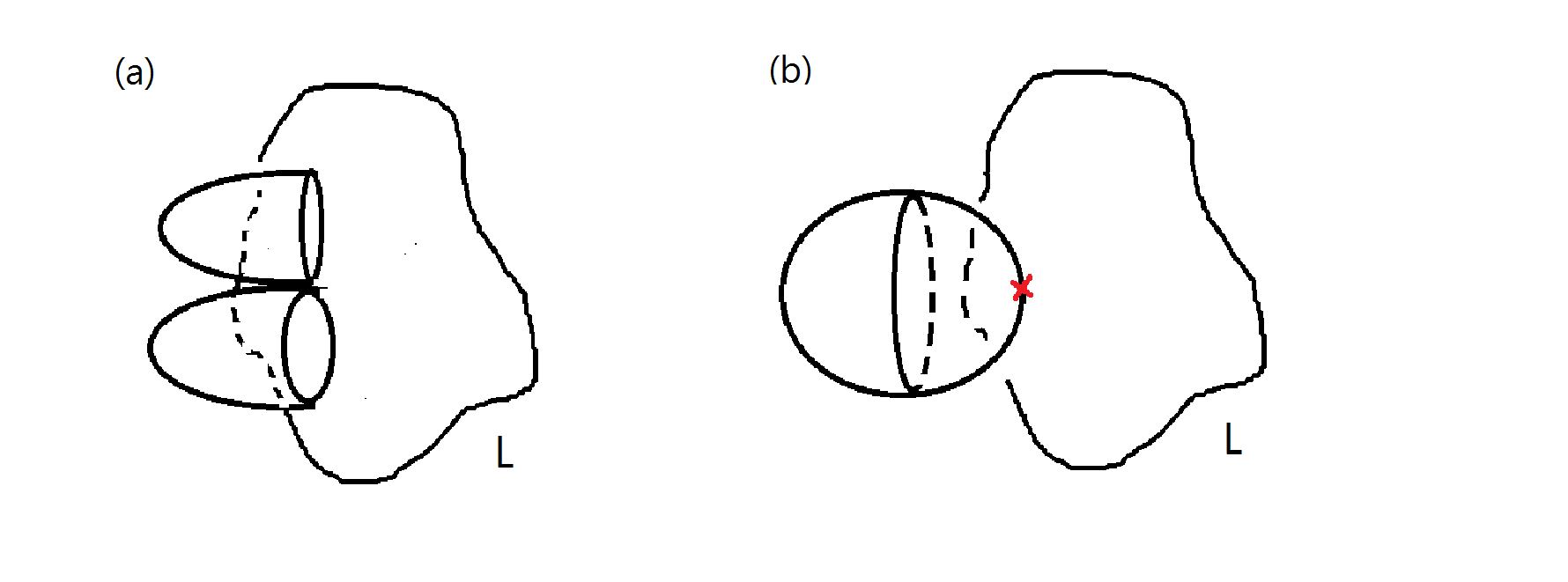}
             \caption{(a)boundary of type I (b)boundary of type II}
             \end{center}
             \end{figure}    
   
   \subsection{Orientation of the Relevant Moduli Spaces}\label{88}
   The moduli space $\mathcal{M}_{\gamma}(\mathfrak{X}^{[\omega]},L)$ is virtually $\mathbb{R}\times \mathcal{M}_{\gamma}(X_{\vartheta},L)$ locally. Since the first factor is the tangent of $S^1$ in the twistor with the positive orientation, it suffices to orient the moduli space $\mathcal{M}_{\gamma}(X_{\vartheta},L)$. By the Theorem 8.1.1 \cite{FOOO}, it suffices to choose a relative spin structure of $L$ to determine a coherent orientation on the moduli spaces $\mathcal{M}_{\gamma}(\mathfrak{X}^{[\omega]},L)$. Since $TL$ is nontrivial and $w_2(TL)\neq 0$, there is no canonical spin structure we can choose unlike the case studied in \cite{L4}. We will going to use $V=\mathcal{O}_{K3}(L)$ and thus $TL\oplus V|_L$ is a trivial oriented rank four bundle and admits a trivial spin structure. This gives the orientation of the moduli space of  $\mathcal{M}_{\gamma}(\mathfrak{X}^{[\omega]},L)$.

\section{The Reduced Open Gromov-Witten Invariants for Rigid Boundary Condition}
Let $\mathbb{L}_{K3}$ be the K3 lattice, i.e.
    \begin{align*}
       \mathbb{L}_{K3}= (-E_8)\oplus (-E_8) \oplus \big(\begin{smallmatrix}
             0 & 1\\
             1 & 0
           \end{smallmatrix}\big) \oplus \big(\begin{smallmatrix}
                        0 & 1\\
                        1 & 0
                      \end{smallmatrix}\big) \oplus \big(\begin{smallmatrix}
                                   0 & 1\\
                                   1 & 0
                                 \end{smallmatrix}\big)  . 
    \end{align*} Here $E_8$ is the unique even, unimodular, positive definite lattice of rank 8. 
\begin{definition}
A marked K3 surface is a pair $(X,\alpha)$ where $X$ is a K3 surface and an isomorphism $\alpha:H^2(X,\mathbb{Z})\rightarrow \mathbb{L}_{K3}$.
\end{definition}
A marking $\alpha$ will induce an identification $\alpha^{\vee}:(\mathbb{L}_{K3})^{\vee}\cong H_2(X,\mathbb{Z})$.
Let $[L]\in {(\mathbb{L}_{K3})}^{\vee}$ with $[L]^2=-2$. Let $\mathcal{M}_{[L]}$ to be the moduli space of K3 surface with $[L]$ can be represented by holomorphic cycles. More precisely,
   \begin{align*}
     \mathcal{M}_{[L]}:=\{(X,\alpha) & \mbox{ marked  K3 surface}| \alpha^{\vee}([L])\wedge \Omega_{X}=0,  \\ 
       &\mbox{ where $\Omega_{X'}$ is a holomorphic $(2,0)$-form of $X'$}\}/ \sim.
   \end{align*} Here two marked K3 surfaces $(X',\alpha')$ and $(X'',\alpha'')$ are equivalent if there is a diffeomorphism $f:X'\rightarrow X''$ such that $\alpha'\circ f^*=\alpha''$. The moduli space $\mathcal{M}_{[L]}$ is a divisor in the moduli space of marked K3 surfaces. Assuming that on a Zariski open subset of $\mathcal{M}_{[L]}$, which is denoted by $\mathcal{M}'_{[L]}$, the homology class $[L]$ via the marking can be realized as a unique smooth rational curve, which we also denote it by $L$ by abuse of notation, in the corresponding K3 surface. We will omit the marking for the simplicity of the notations if there is no confusion. For each $\tilde{\gamma}\in \mathbb{L}_{K3}$, there is a non-empty divisor 
      \begin{align*}
         \mathcal{D}_{\tilde{\gamma}}:=\mathcal{M}'_{[L]}\cap \mathcal{M}_{\tilde{\gamma}}
      \end{align*} in the moduli space $\mathcal{M}'_{[L]}$. 
    Indeed, notice that $\mathcal{M}_{[L]}\cap \mathcal{M}_{\tilde{\gamma}}$ is the intersection of two hyperplanes in the period domain of marked K3 surfaces, which can be viewed as a hyperplane on the complement of a non-degenerate conic. Thus $\mathcal{M}_{[L]}\cap \mathcal{M}_{\tilde{\gamma}}$ is non-empty. In particular, its dense open subset $\mathcal{D}_{\tilde{\gamma}}=\mathcal{M}'_{[L]}\cap \mathcal{M}_{\tilde{\gamma}}$ is non-empty as well.

   The exact sequence (\ref{46}) glues together to a exact sequence of local systems of lattices over $\mathcal{M}'_{[L]}$. In particular, let 
   \begin{align*}
         \Gamma:=\bigcup_{s\in \mathcal{M}'_{[L]}}H_2(X_s,	L)\cong \mathcal{M}'_{[L]}\times \big(H_2(X)/\mbox{Im}H_2(L)\big)
   \end{align*} be a trivial local system over $\mathcal{M}'_{[L]}$. Choose a local section of holomorphic $(2,0)$ form $\{\Omega_s\}$ over $\mathcal{M}'_{[L]}$, we define the central charge to be 
    \begin{align*}
     Z:\Gamma\cong \mathcal{M}'_{[L]}\times \big(& H_2(X)/\mbox{Im}H_2(L)\big)  \rightarrow \mathbb{C}\\
    &(s,\gamma) \longmapsto Z_{\gamma}(s)= \int_{\gamma}\Omega_s.
    \end{align*}
  Notice that since the choice of $\Omega_s$ might differ by elements of $\mathbb{C}^*$. Thus the central charge $Z$ does depend on the choice of the local section but the locus defined by the equation
    \begin{align}\label{89}
       \mbox{Arg}Z_{\gamma_1}=\mbox{Arg}Z_{\gamma_2} 
    \end{align} does not. We will denote the above locus by 
    \begin{align*}
       W'_{\gamma_1,\gamma_2}:=\{s\in \mathcal{M}_{\alpha[L]}| \mbox{Arg}Z_{\gamma_1}(s)=\mbox{Arg}Z_{\gamma_2}(s)\} \subseteq \mathcal{M}'_{[L]}.
    \end{align*} and we write $ W_{\gamma}':=\bigcup_{\gamma_1+\gamma_2=\gamma} W'_{\gamma_1,\gamma_2}$. Notice that the both sides of the equation (\ref{89}) is pluri-harmonic and thus $W'_{\gamma_1,\gamma_2}$ is of real codimension one and locally separate $\mathcal{M}'_{[L]}$ into to components. Moreover, the wall $W'_{\gamma_1,\gamma_2}$ is a real analytic Zariski closed subset in $\mathcal{M}'_{[L]}$. The union in $W'_{\gamma}$ is always finite due to the Gromov compactness theorem. 
    
    \begin{definition} \label{1}
We say a charge $\gamma\in \Gamma$ is strongly primitive if there does not exist $k \geq 2\in \mathbb{N}$ and nonzero classes $\gamma',\gamma''\in \Gamma$ such that
  \begin{align*}
  \gamma=k\gamma'+\gamma'' \mbox{ and } Z_{\gamma''}=0.
  \end{align*}
  %($\gamma$ is primitive in $M^{\perp}$ in the usual sense).
\end{definition}  
Then we have the following theorem:
\begin{thm}
 Assume the charge $\gamma\in \Gamma$ is strongly primitive, $s=(X,\alpha)\notin W'_{\gamma}$, and $[\omega]$ doesn't fall in the hyperplane above, then the moduli space $\mathcal{M}_{\gamma}(\mathfrak{X}_s^{[\omega]},L)$\footnote{Here we use the subindex $s$ to indicate the dependence of the $S^1$-family moduli spaces on $s$.} has no real codimension one boundary. 
\end{thm}    

In particular, it admits a virtual fundamental cycle $[\mathcal{M}^{[\omega]}_{\gamma}(\mathfrak{X}_s,L)]^{vir}$ \cite{FOOO} and the associated reduced open Gromov-Witten invariant can be defined as
    \begin{align}
     \tilde{\Omega}^{[\omega]}(s,\gamma):=\int_{[\mathcal{M}_{\gamma}(\mathfrak{X}^{[\omega]}_s,L)]^{vir}}1.
    \end{align}

A geometric interpretation of the invariant is the following: if one slightly perturb the $S^1$-family of complex structures to a generic $S^1$-family of almost complex structures, then all the pseudo-holomorphic discs in the new family are immersed and $\tilde{\Omega}^{[\omega]}$ is the signed count the number of immersed pseudo-holomorphic discs in the perturbed $S^1$-family. Another naive geometric point of view of $\tilde{\Omega}^{[\omega]}(\gamma)$ is that it counts the number of "special Lagrangian discs" with boundaries on the holomorphic Lagrangian $L$ in the family $\mathfrak{X}^{[\omega]}$. This might be closer to the point of view of \cite{GMN}. However, we don't have the notion of singular special Lagrangian discs nor a good compactification of the relevant moduli spaces.

In general, the moduli space $\mathcal{M}_{\gamma}^{[\omega]}(\mathfrak{X}_s,L)$ is locally modeled by its Kuranishi charts. Namely, for each point $p\in \mathcal{M}_{\gamma}^{[\omega]}(\mathfrak{X}_x,L)$, there exists a $5$-tuple $(E,V,\Gamma, \psi,s)$, where 
  \begin{enumerate}
     \item $V_p$ is a smooth manifold (with corners) and $E_p$ is a
       smooth vector bundle over $V_p$.
       \item $\Gamma_p$ is a finite group acting on $E_p \rightarrow
       V_p$.
       \item $s_p$ is a $\Gamma_p$-invariant continuous section of
       $E_p$.
       \item $\psi_p:s^{-1}(0)\rightarrow M$ is continuous and induced
       homeomorphism between $s^{-1}(0)/\Gamma_p$ and a neighborhood of
       $p\in M$.
    \end{enumerate}
There are compatibility conditions between charts. We will refer the details of the construction and properties of the Kuranishi structures to the fundamental work of Fukaya-Oh-Ohta-Ono \cite{FOOO}. In general, the section $s_p$ will not intersect transversally with the zero section of $E_p$ and this causes difficulty for defining the Gromov-Witten type invariants. For the purpose of Floer theory, Fukaya-Oh-Ohta-Ono \cite{FOOO} further constructed the perturbed multi-sections for the transversality issue and develop the notion of smooth correspondences \cite{FOOO}\cite{F1}\cite{F2}. We will follow the notation of smooth correspondence in \cite{F1}. Let $\mathcal{M},\mathcal{M}_s,\mathcal{M}_t$ be smooth manifolds and $f_s:\mathcal{M}\rightarrow \mathcal{M}_s$, $f_t:\mathcal{M}\rightarrow \mathcal{M}_t$ be smooth maps. Furthermore, assume that $f_t$ is submersive. Then given a smooth differential form $\rho$ of $\mathcal{M}_s$, we associate a smooth correspondence
   \begin{align*}
      Corr_*(\mathcal{M};f_s,f_t)(\rho):=(f_t)_*(f_s)^*\rho,
   \end{align*} where $(f_t)_*$ denotes the integration along fibres. Fukaya \cite{F1} generalized the notion when $\mathcal{M}$ is a Kuranishi space via the perturbed multi-sections and $f_s$ is weakly submersive. With this technology, we will define the open Gromov-Witten invariants to be 
  \begin{align}
     \tilde{\Omega}^{[\omega]}(s,\gamma):=Corr_*(\mathcal{M}^{[\omega]}_{\gamma}(\mathfrak{X}_s,L);tri,tri)(1).
  \end{align}

The construction of perturbed multi-section and smooth correspondence is delicate and complicated. We will refer the readers to \cite{FOOO}\cite{F1}. There are several artificial choices have to be made during the construction of Kuranishi structures and the perturbed multi-sections. A priori, one might get two different open Gromov-Witten invariants $\tilde{\Omega}^{[\omega]}(s,\gamma)$ and $\tilde{\Omega}^{'[\omega]}(s,\gamma)$
, for any two sets of choices. However, there exists a pseudo-isotopy between the resulting Kuranishi structures. Applying the smooth correspondence to the boundary of the moduli space $\partial\big(\mathcal{M}_{\gamma}(\mathfrak{X}^{[\omega]}_s,L)\times [0,1]\big)$ and together with the Stokes theorem of smooth correspondence\cite{F1}, we get
   \begin{align} \label{911}
     &\tilde{\Omega}^{[\omega]}(s,\gamma)-\tilde{\Omega}^{'[\omega]}(s,\gamma) \\
      =&Corr_*(\partial\mathcal{M}_{\gamma}(\mathfrak{X}^{[\omega]}_s,L)\times [0,1];tri,tri)(1) \notag \\
  =&Corr_*(\partial_I\mathcal{M}_{\gamma}(\mathfrak{X}^{[\omega]}_s,L)\times [0,1];tri,tri)(1)+Corr_*(\partial_{II}\mathcal{M}_{\gamma}(\mathfrak{X}^{[\omega]}_s,L)\times [0,1];tri,tri)(1). \notag 
   \end{align}
We omit the last term of (\ref{911}) by assuming $[\omega]$ does not fall on any of the valid hyperplanes. Therefore, the type II boundary of the moduli space is empty. The contribution of the second term of (\ref{911}) is computed in \cite{L4} and (with the notation in (\ref{912})) admits a factor of $\langle \gamma_1,\gamma_2\rangle$ and thus vanishes. In particular, this indicates that the open Gromov-Witten invariants $\tilde{\Omega}^{[\omega]}(s,\gamma)$ is well-defined if $[\omega]$ does not fall on any of the valid hyperplanes.

Assume there is a $1$-parameter family of hyperK\"ahler structures $(s_t,[\omega_t])$ on K3 surface, where $s_t\in \mathcal{M}'_{[L]}$ and $[\omega_t]$ is a K\"ahler class of the K3 surface parametrized by $s_t$. Let $\cup_{t\in [0,1]}\mathcal{M}_{\gamma}(\mathfrak{X}^{[\omega_t]}_{s_t},L)$ be the moduli space of holomorphic discs in the $2$-parameter family of complex structures. Similarly, the data induces pseudo-isotopy of the Kuranishi structures and together with the Stokes theorem for smooth correspondence implies 
\begin{align*} 
  \tilde{\Omega}^{[\omega_1]}(s_1,\gamma)-\tilde{\Omega}^{[\omega_0]}(s_0,\gamma)= Corr_*(\partial_{II}\big(\bigcup_{t\in[0,1]} \mathcal{M}_{\gamma}(\mathfrak{X}^{[\omega_t]}_{s_t},L)\big);tri,tri)(1).
\end{align*}
  Thus, we reach the following proposition of reduced open Gromov-Witten ivnariants for rigid boundary condition.  
\begin{prop} \label{10}
   Let $\gamma\in H_2(X,L)$ and $s\in \mathcal{M}'_{[L]}$.
   Assume that $[\omega].\tilde{\gamma}\neq 0$ for any lifting $\tilde{\gamma}\in H_2(X,\mathbb{Z})$. Then $\tilde{\Omega}^{[\omega]}(s,\gamma)$ is well-defined and only depends on the path connected components contains $[\omega]$ in the complement of valid hyperplanes in the K\"ahler cone.
\end{prop}    
    
  \subsection{Wall-Crossing of the Discs Invariants}
In this section, we want to study how does the invariants $\Omega^{[\omega]}(s,\gamma)$ changes when we vary the K\"ahler class $[\omega]$. More generally, we assume there is a $1$-parameter family of hyperK\"ahler structures $(s_t,[\omega_t])$ on the K3 surface $X$. 

Since there are only finitely many valid hyperplanes for a fixed relative class $\gamma$, we may just consider the case when there is only a single valid hyperplane labeled by $\tilde{\gamma}$. 
 
\begin{thm} \label{15}
   The wall-crossing formula for crossing the hyperplane labeled by $\tilde{\gamma}$ is given by 
      \begin{align} \label{888}
             \Delta\tilde{\Omega}(\gamma):= \tilde{\Omega}^{[\omega_1]}(s_1,\gamma)-\tilde{\Omega}^{[\omega_0]}(s_0,\gamma)= \pm  ([L]\cdot \tilde{\gamma}) GW_{red}(\tilde{\gamma}),
      \end{align} where $([L]\cdot \tilde{\gamma})$ is the intersection pairing in $\mathbb{L}_{K3}$ and $GW_{red}(\tilde{\gamma})$ denotes the reduced Gromov-Witten invariants associated to $\tilde{\gamma}$. The sign in (\ref{888}) is given by 
      \begin{align}
        \frac{1}{2}\big( \mbox{sgn}(\omega_1\cdot \tilde{\gamma})-\mbox{sgn}(\omega_0\cdot \tilde{\gamma}) \big).
      \end{align}       
\end{thm}
 
\begin{proof}   
 Assume that there is an $1$-parameter family of K\"ahler class $[\omega_t], t\in [0,1]$ goes across a single valid hyperplane labeled by $\tilde{\gamma}$ transversally at $t=t_0$. We will follow the argument in \cite{F2} that there exists a Kuranishi structure for the $1$-parameter family of moduli space of pseudo-holomorphic discs. Moreover, the Kuranishi structure is compatible on the boundaries in the sense that 
    \begin{align}\label{900}
                \partial
               \mathcal{M}_{0,\gamma}&(\mathfrak{X},L)=\bigcup_{\tilde{\gamma}:i_*(\tilde{\gamma})=\gamma}
                                                       \big( \mathcal{M}^{cl}_{1,\tilde{\gamma}} \underset{\frak{X}}{\times}   (L\times S^1_{\vartheta})\big)\notag \\
               &\cup \bigcup_{\substack{\gamma_1+\gamma_2=\gamma,\\ Z_{\gamma_1}/Z_{\gamma_2}\in \mathbb{R}_{>0} }}(\mathcal{M}_{1,\gamma_1}(\mathfrak{X},L)\; {}_{(ev_0,ev_{\vartheta})} \! \times_{(ev_0,ev_{\vartheta})}
               \mathcal{M}_{1,\gamma_2}(\mathfrak{X},L))/\mathbb{Z}_2
            \end{align}  
 the Kuranishi structure on the both sides of (\ref{900}) coincides. Apply the smooth correspondence to both sides of (\ref{900}) and together with the Stokes theorem of smooth correspondence \cite{F1}, we have 
  \begin{align} \label{6} 
     \tilde{\Omega}^{[\omega_1]}(s,\gamma)-\tilde{\Omega}^{[\omega_0]}(s,\gamma)= n(L;\tilde{\gamma};\mathcal{J}),
  \end{align} where $\mathcal{J}$ is a $2$-parameter family of complex structures of K3 surface with parameters $(\vartheta,t)$. If we use $GW_{red}(\tilde{\gamma})$ to denote the reduced Gromov-Witten invariants which only depends on $\langle \tilde{\gamma},\tilde{\gamma}\rangle$. The right hand side of (\ref{6}) is defined to be 
     \begin{align} \label{7}
        n(L;\tilde{\gamma};\mathcal{J}):=&[L]\cap ev_*([\mathcal{M}_{\tilde{\gamma},1}(\mathfrak{X}^{[\omega_t]}_s)]^{vir})\in \mathbb{Q} \notag \\
        &= (D_{\tilde{\gamma}}\cdot \mathcal{J})([L]\cdot \tilde{\gamma}) GW_{red}(\tilde{\gamma}).
     \end{align}
       Here the second equality comes from the compatibility of forgetful map
       \begin{align*}
         \mathfrak{forget}: \mathcal{M}_{\tilde{\gamma},1}(\mathfrak{X}_s)\rightarrow \mathcal{M}_{\tilde{\gamma},0}(\mathfrak{X}_s),
       \end{align*} and divisor axiom for closed reduced Gromov-Witten invariants \cite{BL}\cite{L}.
     The notation $ (D_{\tilde{\gamma}}\cdot \mathcal{J})$ denotes the local topological of the surface $\mathcal{J}$.

From our definition, $\mathcal{J}$ has the orientation $\{\frac{\partial}{\partial \vartheta}, \frac{\partial}{\partial t}\}$. The intersection of $\mathcal{J}$ and $D_{\tilde{\gamma}}$ is transverse. Indeed, the complex structure parametrized by $(\vartheta,t)\in \mathcal{J}$ is given by the holomorphic $2$-form $\omega_t-i\mbox{Re}(e^{i\vartheta}\Omega)$. From the definition of $D_{\tilde{\gamma}}$, it suffices to show that $\frac{d}{dt}[\omega_t]|_{t=0}. \tilde{\gamma}\neq 0$, which we can achieved by choose the path $[\omega_t]$ transverse to the valid hyperplane. Since the intersection is transversal, the topological number is just $\pm 1$. Since the jump of the invariants are the same as long as the path we choose is transverse to the valid hyperplane from Proposition \ref{10}, we may choose the family of K\"ahler class $[\omega_t]$ such that $\frac{d}{dt}\omega_t|_{t=0}=-\mbox{Re}\Omega$ without loss of generality. With this choice, the orientated tangent plane of $\mathcal{J}$ at the intersection point with $D_{\tilde{\gamma}}$ coincide with that of the twistor line from the hyperK\"ahler pair $(\omega,\Omega)$. In particular, if we choose $[\omega_t]$ such that $([\omega_t]\cdot \tilde{\gamma})$ is decreasing then the topological intersection would be just $1$, which verify the sign in (\ref{888})
\end{proof}
   
%Claim: For every pair $(X,L)$, one can deform it to an elliptic surface with $C$ %as a section.
% The following is an extension lemma of rational curves in K3 %surfaces: 
%\begin{lem}\cite{R3}\cite{BHT}
%For every hyperplane labeled by $\tilde{\gamma}$, %$\tilde{\gamma}^2\geq -2$, then the homology class $\tilde{\gamma}$ %can be realized as holomorphic cycle in the family %$\{\mathfrak{X}_s\}$. 
%\end{lem} 
%\begin{proof}
%   The holomorphic volume form of $\mathfrak{X}_{s,\vartheta}$ is %given by $\omega+i\mbox{Im}(e^{-i\vartheta}\Omega)$ and has zero %pairing with $\tilde{\gamma}$ for exactly two value of $\vartheta$ %(differed by $\pi$). In particular, the homology class %$\tilde{\gamma}$ is of type $(1,1)$ with respect to that two complex %structures. The proposition follows from the Riemann-Roch of K3 %surfaces and Serre duality,
%      \begin{align*}
%        %h^0(X,\mathcal{O}(\tilde{\gamma}))+h^0(X,\mathcal{O}(\tilde{\gamma}))\g %eq 2
 %     \end{align*}    
%\end{proof}

\subsection{The Vanishing Theorem}
Let $\gamma\in H_2(X,L)$ and $\tilde{\gamma}\in H_2(X,\mathbb{Z})$ such that $\iota(\tilde{\gamma})=\gamma$\footnote{Here we mean the identity holds for the corresponding class under the identification of marking.}. There exists a non-empty divisor $\mathcal{D}_{\tilde{\gamma}}\subseteq \mathcal{M}'_{[L]}$. For any 
$s_0\in \mathcal{D}_{\tilde{\gamma}}$ and for any choice of K\"ahler class $[\omega]$, any holomorphic discs in $\mathfrak{X}^{[\omega]}_{s_0}$ with boundary on $L$ and relative homology class $\gamma$ has symplectic area 
    \begin{align*}
      |\int_{\gamma}\Omega_{s_0}|=|\int_{\tilde{\gamma}}\Omega_{s_0}|=0.
    \end{align*} In other words, there exists an obvious cohomological constraint of existence of holomorphic discs in the relative class $\gamma$. In particular, we have the vanishing of certain open Gromov-Witten invariants:
    \begin{thm} \label{2001}
    Let $\tilde{\gamma}\in \mathbb{L}_{K3}$ and $\iota(\tilde{\gamma})=\gamma$. Assume that $s_0\in \mathcal{D}_{\tilde{\gamma}}$. Then for any choice of the K\"ahler class $[\omega]$, we have
    \begin{align*}
       \tilde{\Omega}^{[\omega]}(s_0,\gamma)=0.
    \end{align*} 
    \end{thm}

The lemma also provides us the vanishing of some invariants for the K3 surface near $X_{s_0}$.
\begin{prop}
There exists a neighborhood $U$ of $s_0$ in $\mathcal{M}'_{[L]}$ such that any point $s_1\in U$, there is a path of K\"ahler class $[\omega_t]$ such that it doesn't hit any valid hyperplanes. In particular, we have 
   \begin{align*}
      \tilde{\Omega}^{[\omega_1]}(\gamma,s_1)=\tilde{\Omega}^{[\omega_0]}(\gamma,s_0)=0.
   \end{align*}
\end{prop}   
\begin{proof}
  There are only finitely many valid hyperplanes which are a priori determined cohomologically. Thus, the first statement follows from the fact that any local complex deformation of a K\"ahler manifold is still K\"ahler. Thus, there exists a neighborhood $U\ni s_0$ such that every point $s_1\in U$ there is a $1$-parameter family of K\"ahler class $[\omega_t]$ of $X_{s_t}$, $t\in[0,1]$. Since the K\"ahler cone of $X_{s_t}$ is open in $H^{1,1}(X_{s_t})$, there exists a open neighborhood of $\{[\omega_t]\}$ in the union of $H^{1,1}$. Therefore, one can perturb it to avoid finitely many valid hyperplanes. The second statement follows from the first part and the cobordism argument.
\end{proof}

 Together with the Theorem \ref{15}, we have a closed formula for the reduced open Gromov-Witten invariants:
 \begin{thm}
     Let $\gamma\in H_2(X,L)$, then
     \begin{align} \label{1005}
       \tilde{\Omega}^{[\omega]}(\gamma)=\sum_{\tilde{\gamma}:\iota(\tilde{\gamma})=\gamma} \pm ([L]\cdot \tilde{\gamma}) GW_{red}(\tilde{\gamma}), 
     \end{align} where the sign is the one given by Theorem \ref{15}. In particular, the reduced open Gromov-Witten invariants are combination of closed reduced Gromov-Witten invariants. 
     
 \end{thm} Notice that the terms on the right hand side of (\ref{1005}) are non-zero only when $\tilde{\gamma}^2\geq -2$ and there are only finitely many such $\tilde{\gamma}$.

\subsection{Multiple Cover Formula and Integrality Conjecture}
In \cite{YZ}(See also \cite{B2}), Yau and Zaslow derived the following intriguing Yau-Zaslow formula:
    \begin{align}\label{8}
     \frac{q}{\Delta(q)}=\prod_{k>0}\frac{1}{(1-q^k)^{24}}=&\sum_{d\geq 0} G_d q^d \\    
                                              =1+&24q+324q^2+3200q^3+25650q^4+176256q^5+\cdots .            \nonumber    
    \end{align} The integers $G_d$ exactly count the number of nodal rational curves in a generic algebraic K3 surface of genus $d$ and in the linear system of its natural polarization. Notice that the curve classes in the above linear system are always primitive. The formula (\ref{8}) motivates the study of the general theory of reduced Gromov-Witten invariants on K3 surfaces \cite{BL}\cite{L}\cite{MP}. 
    
    When the curve classes is non-primitive, the genus zero reduced Gromov-Witten invariants are more complicated to compute. However, people find out that the genus zero reduced Gromov-Witten invariants can always be computed via the integers $G_d$ the self-intersection number of the curve classes (See also \cite{MP} for multiple cover formula for the higher genus reduced Gromov-Witten invariants):
    \begin{thm} \cite{GV}\cite{MP} Let $\mathbb{L}_{K3}$ be the K3 lattice. Given a curve class $\beta\in \mathbb{L}_{K3}$ and denote the genus zero reduced Gromov-Witten invariant associated to the Poincar\'e dual of $\beta$ by $n_{\beta}$. Then 
          \begin{align}\label{16}
              n_{\beta}=\sum_d \frac{1}{d^3}G_{\frac{1}{2}(\frac{\beta}{d})^2+1}.   
          \end{align} Here $(\beta/d)^2$ denotes the natural extension of self-intersection pairing on $\mathbb{L}_{K3}\otimes \mathbb{R}$ and we set $G_k=0$ if $k$ is not an integer.
    \end{thm} We will call (\ref{16}) the multiple cover formula for genus zero reduced Gromov-Witten invariants. First, the formula (\ref{16}) indicates that the genus zero reduced Gromov-Witten invariant $n_{\beta}$ only depends on $\langle \beta,\beta\rangle$. Secondly, all the reduced Gromov-Witten invariants after a mysterious transformation involves $1/d^3$ will produce integer invariants. It is interesting to understand the geometric (or enumerative) meaning of these integer-valued invariants.

    Notice that the magic number $1/d^3$ for the multiple cover formula of genus zero reduced Gromov-Witten invariants is the same as the one appears in the multiple cover formula (\ref{998}) for genus zero Gromov-Witten invariants in Calabi-Yau $3$-folds. This is because the tangent-obstruction theory of reduced Gromov-Witten invariants on K3 surfaces is similar to that of Gromov-Witten theory on Calabi-Yau $3$-folds.

      Multiple cover formula for holomorphic discs is not well-studied mainly because there are not many cases the open Gromov-Witten invariants can be defined. The first one is computed in the holomorphic discs in the total space of $\mathcal{O}_{\mathbb{P}^1}(-1)\oplus \mathcal{O}_{\mathbb{P}^1}(-1)$, which admits a real structure and a torus action \cite{KS3}\cite{GZ}. 
      It is speculated that there is also a multiple cover formula for disc counting in Calabi-Yau $3$-folds, with $1/d^3$ replaced by $1/d^2$ (up to a sign) \cite{FOOO}\cite{F2}. The general philosophy is that the reduced theory of K3 surfaces is similar to the theory on Calabi-Yau $3$-folds. Therefore, it is reasonable to expect that the reduced open Gromov-Witten invariants on K3 surfaces share similar multiple cover formulas. 
      \begin{conj} \cite{L4} \label{13} For any choice of the K\"ahler class $[\omega]$, there exists integers $\{\Omega^{[\omega]}(\gamma)\in \mathbb{Z}\}$ such that 
         \begin{align*}
           \tilde{\Omega}^{[\omega]}(\gamma)=\sum_d \pm\frac{1}{d^2}\Omega^{[\omega]}(\frac{\gamma}{d}).
         \end{align*}
      \end{conj}Naively, $\Omega^{[\omega]}(\gamma)$ counts the immersed holomorphic discs in the K3 surfaces. The conjecture is verified for the Lefschetz thimbles in the Ooguri-Vafa space:
  \begin{thm}\cite{L4} Let $L_u$ be an elliptic fibre of the Ooguri-Vafa space $X$ and $\gamma\in H_2(X,L_u)$ represents the Lefschetz thimble. Its corresponding open Gromov-Witten invariant is independent of the choice of $[\omega]$ and is given by
       \begin{align*}
                \tilde{\Omega}(\gamma,L_u)=\begin{cases} \frac{(-1)^{d-1}}{d^2}& \text{, if $\gamma=  d\gamma_e$, $d\in \mathbb{Z}$}\\
                       0& \text{, otherwise}.
                                                                 \end{cases}
               \end{align*}In particular, it suggests that
       \begin{align*}
             \Omega(\gamma,L_u)=\begin{cases} 1& \text{if $\gamma= \pm\gamma_e,$}\\
                    0& \text{otherwise.}
                                                              \end{cases}
            \end{align*}                   
  \end{thm}
 
Here we are going to prove the multiple cover formula and the integrality conjecture for reduced open Gromov-Witten invariants with rigid boundary conditions.
\begin{thm} \label{14}
 Assume that $L$ is a smooth rational curve in a K3 surface $X$. Then there exists integers $\{\Omega^{[\omega]}(\gamma)\in \mathbb{Z}\}$ such that 
     \begin{align}
      \tilde{\Omega}^{[\omega]}(\gamma)=\sum_d \frac{1}{d^2}\Omega^{[\omega]}(\frac{\gamma}{d}),
     \end{align} if $[\omega].\tilde{\gamma}\neq 0$, for every $\tilde{\gamma}\in H_2(X,\mathbb{Z})$ such that $\iota(\tilde{\gamma})=\gamma$. Here we use the convention that $\gamma/k=0$ if $\gamma$ is not $k$-divisible.
\end{thm}
\begin{proof}
We will use following easy fact from algebra:
\begin{lem}
  Let $\gamma\neq 0\in H_2(X,L)$ and $\tilde{\gamma}$ be any lifting in $H_2(X)$. If $\tilde{\gamma}$ is $k$-divisible then $\gamma$ is $k$-divisible.
\end{lem}
From Theorem \ref{2001}, given a relative homology class $\gamma$, there exists a K3 surface $(X_s,\alpha)\in \mathcal{M}'_{[L]}$ and a K\"ahler class $[\omega_0]$ such that the open Gromov-Witten invariants vanishes
   \begin{align*}
      \tilde{\Omega}^{[\omega_0]}(\gamma,s)=0.
   \end{align*} In particular, this also implies vanishing of corresponding $\Omega^{[\omega_0]}$ and Theorem \ref{14} follows. For different choices of $[\omega]$, one only need to prove that the wall-crossing term has the similar property. From Theorem \ref{15} and (\ref{16}),
     \begin{align*}
      \Delta\tilde{\Omega}(\gamma)&=\pm ([L].\tilde{\gamma})GW_{red}(\tilde{\gamma}) \\
                                  &=\pm ([L].\tilde{\gamma})\sum_k \frac{1}{k^3}G_{\frac{1}{2}(\frac{\tilde{\gamma}}{k})^2+1}\\
                                  &=\pm \sum_k \frac{1}{k^2}([L].\frac{\tilde{\gamma}}{k})G_{\frac{1}{2}(\frac{\tilde{\gamma}}{k})^2+1}.
     \end{align*}Therefore, 
     \begin{align} 
     \Delta\Omega(\gamma)=\pm  ([L].\frac{\tilde{\gamma}}{k})G_{\frac{1}{2}(\frac{\tilde{\gamma}}{k})^2+1}
     \end{align}are integers from Yau-Zaslow formula and (\ref{8}). This finishes the proof of Theorem \ref{14}. In particular, we see the contribution of $d$-folds multiple cover to the open reduced Gromov-Witten invariants is $1/d^2$. 
     
\end{proof}
A direct consequence of the multiple cover formula is the "reality condition"\footnote{The relevant moduli spaces associated to both sides of (\ref{51}) is diffeomorphic. The corollary can also be derived from directly checking the orientations defined in Section \ref{88} associated to the two moduli spaces are the same.}. Indeed, each $\tilde{\gamma}$ contribute the same to $\tilde{\Omega}^{[\omega]}(\gamma)$ as $-\tilde{\gamma}$ contributes to $\tilde{\Omega}^{[\omega]}(-\gamma)$. Therefore, we have the following "reality condition":
     \begin{cor} \label{reality}
       Let $\gamma\in H_2(X,L)$ be a relative class, then 
         \begin{align} \label{51}
           \tilde{\Omega}^{[\omega]}(s,-\gamma)=\tilde{\Omega}^{[\omega]}(s,\gamma).   \end{align}
     \end{cor}

%\begin{prop}
%  Let $f_t:(D^2,\partial D^2)\rightarrow (\mathfrak{X}^{[\omega_t]},L), t\in %[0,1)$ be a family of holomorphic  discs. Assume that the $\lim_{t\rightarrow %1}f_t$ converges to an immersed rational curve of homology class $\tilde{\gamma}$ %with a point on $L$. If $\tilde{\gamma}$ is primitive, and %$[\omega_t].\tilde{\gamma}\neq 0 \mbox{ for } t\neq 0$, then $f_t$ can not be %multiple covered. 
%\end{prop}
%\begin{proof}
%  
%\end{proof}
	
\begin{rmk}
  For the case that the curve class $\beta\in H_2(X)$ is exactly divisible by two, the genus zero multiple cover formula for reduced Gromov-Witten invariants is 
    \begin{align*}
       N_{\beta}=G_{4g-3}+\frac{1}{8}G_g.
    \end{align*}
  Wu \cite{W3} proved that the integer $G_{4g-3}$ exactly counts the number of genus zero stable maps $f:C\rightarrow X$ with $f_*([C])=\beta$ and not a double cover factoring through the normalization of $f(C)$, while $G_g$ is the number of genus zero stable maps with the image curve of homology class $\frac{1}{2}\beta$. From the Theorem \ref{15} and Theorem \ref{14}, there is also a similar enumerative interpretation for the multiple cover formula of the reduced open Gromov-Witten invariants: for a relative class $\gamma \in H_2(X,L)$ and $\tilde{\gamma}\in H_2(X)$ such that $\iota({\tilde{\gamma}})=\gamma$ is exactly $2$-divisible, then the wall-crossing term $\Delta \tilde{\Omega}(\gamma)$ satisfies the multiple cover formula 
     \begin{align*}
        \Delta \tilde{\Omega}(\gamma)= \Delta\Omega(\gamma)+ \frac{1}{4}\Delta\Omega(\frac{1}{2}\gamma).
     \end{align*} The integer $\Delta  \Omega(\frac{1}{2}\gamma)$ counts the number of difference of the number of stable pseudo-holomorphic discs in the relative class $\frac{1}{2}\gamma$ after a generic perturbation of the almost complex structures of the $S^1$-family of complex structures. The integer $\Delta\Omega(\gamma)$ counts the difference of the number of pseudo-holomorphic discs in the relative class $\gamma$ which do not factor through the normalization of their images when $[\omega]$ varies across the valid layer labeled by $\tilde{\gamma}$ after a generic perturbation. It is interesting to ask the enumerative meaning of the multiple cover formula of reduced open Gromov-Witten invariants for the general situation.
\end{rmk}

Department of Mathematics, Stanford University\\
 E-mail address: yslin221@stanford.edu

\end{document}